\documentclass[11pt]{amsart}
\usepackage{pdfpages}
\usepackage{booktabs}
\usepackage{hyperref}
\usepackage{float}
\usepackage{graphicx, import}
\graphicspath{ {./images/} }
\usepackage[latin1]{inputenc}
\usepackage{amsmath}
\usepackage{amsfonts}
\usepackage{amssymb}
\usepackage{amsthm}
\usepackage{amscd}
\usepackage{verbatim}
\usepackage{subfigure}
\usepackage{caption}
\usepackage{pinlabel}
\usepackage{stmaryrd}
\usepackage{enumerate, enumitem}
\usepackage{todonotes}
\usepackage{bm}
\usepackage{thmtools}
\usepackage{thm-restate}
\usepackage{lipsum}
\usepackage{setspace}
\usepackage{mathtools}
\usepackage[abs]{overpic}
\usepackage{fancyvrb}
\usepackage{float}
\usepackage{mathdots}
\usepackage{tikz}
\usetikzlibrary{arrows}
\usetikzlibrary{decorations.pathreplacing}
\usepackage{verbatim}
\usetikzlibrary{cd}
\tikzset{taar/.style={double, double equal sign distance, -implies}}
\tikzset{amar/.style={->, dotted}}
\tikzset{dmar/.style={->, dashed}}
\tikzset{aar/.style={->, very thick}}
\usepackage[section]{placeins}

\newtheorem{theorem}{Theorem}[section]

\newtheorem{lemma}[theorem]{Lemma}
\newtheorem{proposition}[theorem]{Proposition}

\newtheorem{corollary}[theorem]{Corollary}

\theoremstyle{definition}
\newtheorem{definition}[theorem]{Definition}

\theoremstyle{remark}
\newtheorem{remark}[theorem]{Remark}

\def \Hom{\textrm{Hom}}

\def\RP{\mathbb{RP}^3}
\def\F{\mathbb{F}}
\def\N{\mathbb{N}}
\def\C{\mathbb{C}}
\def\Q{\mathbb{Q}}
\def\R{\mathbb{R}}
\def\Z{\mathbb{Z}}

\def\M{M^{\textrm{inv}}}
\def\Lo{L_0^{\textrm{inv}}}
\def\Li{L_1^{\textrm{inv}}}
\def \talpha{\mathbb{T}_{\bm{\alpha}}}
\def \Sym{\textrm{Sym}}
\def \tbeta{\mathbb{T}_{\bm{\beta}}}

\def\HFK {\mathit{HFK}}

\newcommand\HFKhat{\widehat{\HFK}}

\author[Timothy Bates, Aakash Parikh]{Timothy Bates and Aakash Parikh}
\email{ap1792@math.rutgers.edu}
\email{tim.bates@math.rutgers.edu}
\address{Rutgers University, New Brunswick, NJ, USA}
\thanks{T. Bates was partially supported by NSF Grant DMS-2505573. A. Parikh was partially supported by the Air Force Office of Scientific Research
under awards numbered FA9550-23-1-0011 and FA9550-23-1-0400}
\numberwithin{equation}{section}
\title{A note on the knot Floer homology of freely 2-periodic knots and their quotients}
\begin{document}
\begin{abstract}  
A knot P in the three-sphere is freely 2-periodic if it is preserved setwise by a free order-two action. There is a natural projective quotient knot associated to P. We establish a rank inequality between the knot Floer homologies of P and its quotient as a consequence of Large's generalization of Seidel--Smith's localization spectral sequence associated to order 2 actions in Lagrangian Floer homology. As a corollary we obtain an inequality between the Seifert genus of P and the rational Seifert genus of its quotient. We also implement a program which computes the E2 page of this spectral sequence using a modification of Baldwin--Gillam's grid homology calculator. 
\end{abstract}
\maketitle
\section{Introduction}
\subsection{Main results}\label{mainresults}
A \emph{symmetric knot} $(P,\tau)$ is a knot $P\subset S^3$ along with a finite-order orientation-preserving diffeomorphism $\tau\colon (S^3,P)\to (S^3,P)$. A symmetric knot $P$ is \emph{freely 2-periodic} if $\tau$ is free of order 2. The quotient of $S^3$ by such an action is $\RP$. We fix $(P,\tau)$ to mean a freely 2-periodic knot from here and often omit $\tau$ from the notation. If $q\colon S^3\to S^3/\tau\cong\RP$ denotes the quotient map, then $\overline{P}:=q(P)$ is a knot in $\RP$ called the \emph{quotient knot of} $P$. The quotient knot $\overline{P}$ is a class $1$ knot, which means $[\overline{P}]$ generates $ H_1(\RP;\Z)\cong \Z/2\Z$. In Figure \ref{fig:10157} the freely 2-periodic symmetry of $10_{157}$ and its resulting quotient knot are displayed.
\begin{figure}[H]
    \centering
    \includegraphics[width=0.4\linewidth]{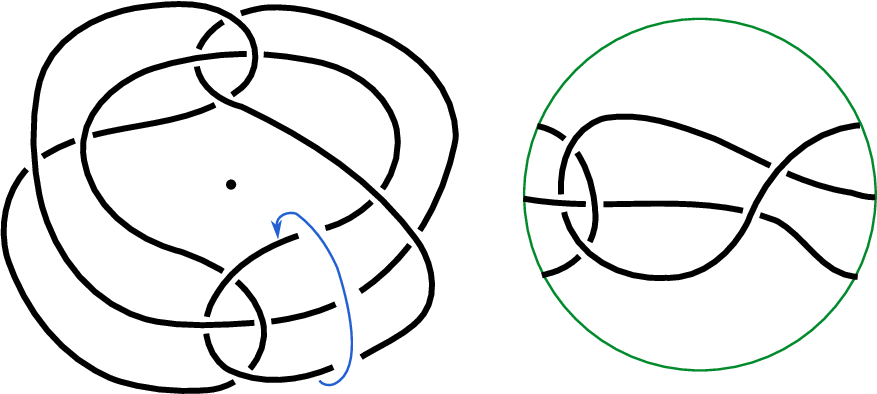}
    \caption{Left: The freely 2-periodic knot $P=10_{157}$. The symmetry $\tau$ of $P$ can be seen as rotation by 180 about the center dot followed by rotation by 180 in the direction of the blue curved arrow. Right: The quotient knot $\overline{P}\subset \mathbb{RP}^3$ in the disk model for $\mathbb{RP}^2$.}
    \label{fig:10157}
\end{figure}
Throughout this paper we will take $P$ to be a knot in $S^3$, $\overline{P}$ a class $1$ knot in $\RP$, and $K$ a knot in an arbitrary closed three-manifold $Y$.

Knot Floer homology is a suite of knot invariants due to Ozsv\'{a}th and Szab\'{o} \cite{OSknots} and independently Rasmussen \cite{RasmussenThesis}. We say that a knot $K$ in a three-manifold $Y$ is \emph{rationally nullhomologous} if $[K]=0\in H_1(Y,\Q)$. The version of knot Floer homology we use in this paper assigns to a rationally nullhomologous knot $K\subset Y$ a $\F$--vector space denoted $\HFKhat(Y,K)$, where throughout $\F:=\F_2$ is the field of two elements. Knot Floer homology for rationally nullhomologous knots was first defined in \cite{OSrational}. Notice that any knot in a rational homology $3$-sphere such as $\RP$ or $S^3$ is rationally nullhomologous. 

Knot Floer homology decomposes as a direct sum on several gradings. There is a homological (or Maslov) grading, and a relative $\textrm{Spin}^c$ grading which further gives rise to a $\textrm{Spin}^\textrm{c}$ grading and an Alexander grading. See Section \ref{gradings} for more background on $\textrm{Spin}^{c}$ and Alexander gradings in knot Floer homology, particularly for rationally nullhomologous knots, but here we summarize the salient points. There is only one $\textrm{Spin}^\textrm{c}$ structure on $S^3$ and the Maslov and Alexander gradings are both $\Z$ valued, so for knots $P\subset S^3$ we have the decomposition \[\HFKhat(S^3,P)=\bigoplus_{(M,A)\in\Z^2}\HFKhat_M(P,A)\] where $M$ is the Maslov grading and $A$ is the Alexander grading. For a class 1 knot $\overline{P}\subset \RP$, the situation is as follows. There are $\frac{1}{2}\Z+\frac{1}{4}$ valued Maslov and Alexander gradings, again denoted by $M$ and $A$ respectively. In addition, there are two $\textrm{Spin}^\textrm{c}$ structures on $\RP$, which we denote by $\textrm{Spin}^\textrm{c}(\RP)=\{\mathfrak{s}_0,\mathfrak{s}_1\}$. So we have the decomposition \[\HFKhat(\RP,\overline{P})=\bigoplus_{M,A\in\frac{1}{2}\Z-\frac{1}{4}} \HFKhat_M(\RP,\overline{P},A;\mathfrak{s}_0)\oplus \bigoplus_{M,A\in\frac{1}{2}\Z+\frac{1}{4}} \HFKhat_M(\RP,\overline{P},A;\mathfrak{s}_1).\] We often collapse one or more of the gradings. For instance, $\HFKhat(Y,K,A)\coloneq\bigoplus_{d, \mathfrak{s}}\HFKhat_d(Y,K,A;\mathfrak{s})$.

With this notation in place, the main result of this paper is the following spectral sequence in knot Floer homology for a freely 2-periodic knot $(P,\tau)$ with quotient $\overline{P}\subset \RP$. 
\begin{theorem}\label{mainthm}
    There is a spectral sequence with $E^1$ page equal to $\HFKhat(S^3,P)\otimes V\otimes \F[\theta,\theta^{-1}]$ and whose $E^\infty$ page is isomorphic to $\HFKhat(\RP,\overline{P})\otimes \F[\theta,\theta^{-1}]$. Furthermore, this spectral sequence splits along Alexander gradings, carrying grading $A$ to grading $\frac{2A+1}{4}$ for all integers $A$.
\end{theorem}
 In the above theorem, $V$ denotes the bigraded vector space $\F_{(0,0)}\oplus \F_{(-1,-1)}$; that is the two dimensional vector space over $\mathbb{F}$ with homogeneously graded basis elements in bigradings $(d,s)=(0,0)$ and $(-1,-1)$. 
 \begin{remark}
    Using \cite[~Theorem 1.1]{Hendricks_2016}, it is straightforward to show that every page of the spectral sequence of Theorem \ref{mainthm} is an invariant of $(P,\tau)$ in the sense that given other choices in the construction (of Heegaard diagram, almost complex structure etc.), there is an isomorphism between each page of the resulting spectral sequences. See also Remark \cite[~4.13]{hendricks2022rank}.
\end{remark}
A \emph{$q$-periodic knot} $(K,\tau)$ is a symmetric knot where $\tau$ is orientation preserving and has a non-empty fixed set. Let $K$ be a $q$-periodic knot with quotient $\overline{K}$ for $q=p^r$ some prime power. Edmonds proved using minimal surface theory that there exists an equivariant minimal genus Seifert surface for $K$, which in the case of $q=2$ specializes to the fact there is an inequality \[g(K)\ge g(\overline K)+\frac{\lambda-1}{2}\] where $\lambda$ is the linking number of $K$ with the fixed axis \cite[~Theorem 4]{MR769284}. Work of Hendricks, along with knot Floer homology genus detection results \cite[~Theorem 1.2]{OSgenus}, recovers Edmonds' result \cite[~Corollary 1.6]{HendricksDP}. From Theorem \ref{mainthm}, we establish an analog of Edmonds' condition for freely 2-periodic knots. Firstly, we remind the reader that there is an appropriate notion of genera for rationally nullhomologous knots. If \(K\) is a rationally nullhomologous knot in a closed oriented \(3\)-manifold \(Y\), a
\(p\)-Seifert surface for \(K\) is a properly embedded oriented surface
\[
        F\subset Y\setminus \operatorname{int}(\nu K)
\]
whose boundary represents \(p\) times the longitude of \(K\), or equivalently
whose boundary wraps algebraically \(p\) times around \(K\). The rational genus
of \(K\), in the sense of Calegari--Gordon, is
\[
        \|K\|
        =
        \inf_F \frac{-\chi(F)}{2p},
\]
where the infimum is taken over all $p$ and all \(p\)-Seifert surfaces \(F\) for \(K\) \cite[~Definition 2.3]{MR3008914}.

\begin{corollary}\label{cor:rational-genus-bound}
Let \(g(P)\) denote the Seifert genus of
\(P\), and let \(\|\overline P\|\) denote the rational genus of \(\overline P\). Then 
\[
        2\|\overline P\|+\frac{1}{2}\le g(P).
\]
\end{corollary}

\begin{proof}

Let
\[
        A_{\max}(\overline P)
        =
        \max\left\{
        A\in \mathbb Q \mid
        \widehat{HFK}(\mathbb{RP}^3,\overline P,A)\neq 0
        \right\}
\]
and 
\[
        A_{\max}(P)
        =
        \max\left\{
        A\in \mathbb Z \mid
        \widehat{HFK}(S^3,P,A)\neq 0
        \right\}.
\]
By Ni's Thurston norm detection theorem for link Floer homology in rational
homology spheres \cite[~Theorem 3.1]{MR2546619}, we have
$\|\overline P\|=A_{\max}(\overline P)-\frac{1}{2},$ and by Ozsv\'{a}th-Szab\'{o}'s genus detection result \cite[~Theorem 1.2]{OSgenus} we have $g(P)=A_{\max}(P)$. By Theorem \ref{mainthm}, $\frac{2A_{\max}(P)+1}{4}\ge A_{\max}(\overline P).$ 
\end{proof}

Comparing the ranks of the $E^1$ and $E^\infty$ pages of the spectral sequence in Theorem \ref{mainthm} we also obtain the following corollaries.
\begin{corollary}
For every $A\in \Z$ there is a rank inequality\begin{equation}\label{rankinequality} rk(\HFKhat(S^3,P,A))+rk(\HFKhat(S^3,P, A+1))\ge rk\left(\HFKhat\left(\RP, \overline{P},\frac{2A+1}{4}\right)\right). \end{equation}
\end{corollary}

\begin{corollary}
For $i=0,1$ there is a rank inequality \begin{equation} rk(\HFKhat(S^3,P))\ge rk(\HFKhat(\mathbb{RP}^3,\overline{P};\mathfrak{s}_i)).\end{equation}
\end{corollary}
\begin{proof}
    Take the sum of Equation \ref{rankinequality} over all $A$ of one parity. For a knot $\overline{P}\subset \RP$, according to Lemma \ref{gradingdiffs} we have that \[\bigoplus_{A\in\mathbb{Z}}\HFKhat(\RP,\overline{P},\frac{4A+1}{4})=\HFKhat(\RP,\overline{P};\mathfrak{s_0})\]
    and 
    \[\bigoplus_{A\in\Z}\HFKhat(\RP,\overline{P},\frac{4A+3}{4})=\HFKhat(\RP,\overline{P};\mathfrak{s_1}).\]
\end{proof}

Theorem \ref{mainthm} belongs to a growing body of work on spectral sequences in Heegaard Floer homology associated to involutions of knots and three-manifolds. In particular, previous work similar to Theorem \ref{mainthm} focuses on the better studied cousins of freely 2-periodic knots, the strongly invertible and doubly periodic knots -- knots which are preserved by an orientation-preserving order 2 action on $S^3$ with fixed set equal to an unknotted axis. Hendricks in \cite{HendricksDP} used a localization theorem in Lagrangian Floer homology due to Seidel--Smith's (\cite{SS}) to establish a rank inequality for the knot Floer homology of doubly periodic knots, a result which was refined by Hendricks--Lipshitz--Sarkar in \cite{Hendricks_2016} and Boyle in \cite{Boyle2022rankinequalitiesonknotfloerhomologyofperiodicknots}. The second author used a generalization of Seidel--Smith's localization theorem due to Large to establish similar results for strongly invertible knots \cite{parikh2024}. Moreover, Theorem \ref{mainthm} should be compared to Lidman--Manolescu's work \cite{LidmanManolescu} which establishes a rank inequality for the Floer homology of p-fold covers of rational homology spheres using the Seiberg--Witten Floer stable homotopy type. Theorem \ref{mainthm} should also be compared to  Large's work on the behavior of Heegaard Floer homology of three manifolds under double covers which lift to a $\Z$-fold cover \cite{Large} (which notably excludes the double cover $S^3 \rightarrow \mathbb{RP}^3$). Hendricks established a rank inequality in knot Floer homology for the branched double cover of a knot in $S^3,$ a result which was extended by Large to branched double covers of nullhomologous knots in arbitrary closed oriented three-manifolds \cite{Hendricks2012rank,Large}.  Hendricks--Lidman--Lipshitz also deduced similar rank inequalities in Heegaard Floer homology for branched double covers \cite{Hendricks2012rank, hendricks2022rank}. The main purpose of this paper is to fill a small gap in this literature; namely, establishing a rank inequality in knot Floer homology associated to the lift of a knot under an ordinary double cover. We focus our attention on the case of the double cover $S^3\to \RP$ but expect that our results can be generalized.

A freely 2-periodic knot $P$ differs from a strongly invertible or doubly periodic knot because $P$ does not admit a diagram where the symmetry is seen as rotation about one axis, but nevertheless it is possible to construct an equivariant Heegaard diagram for $P$ and a quotient Heegaard diagram for $\overline{P}$; see Subsection \ref{sympring} for the construction and Figures \ref{fig:10157heeg} and \ref{fig:quotheeg} for an example thereof. Equivariant Heegaard diagrams such as these are necessary to construct localization spectral sequences in Heegaard Floer and knot Floer homology from involutions on the three-manifolds or knots. The technical tool used to produce these localization spectral sequences is again Large's generalization of Seidel--Smith's localization theorem \cite{Large,SS}, which associates a spectral sequence in Lagrangian Floer homology to pairs of Lagrangians in a symplectic manifold equipped with an involution satisfying bundle theoretic conditions (see Definition \ref{tangentnormalisodef}). Most of the work in establishing Theorem \ref{mainthm} consists of constructing the equivariant Heegaard diagrams and verifying these restrictive hypotheses.

\subsection*{Organization} This paper is organized as follows. Section \ref{background} contains background material on freely 2-periodic knots and knot Floer homology. In Section \ref{diagrams} we construct equivariant Heegaard diagrams for freely 2-periodic knots and their quotients, and analyze gradings on the knot Floer complexes of the aforementioned. The existence of the spectral sequence of Theorem \ref{mainthm} is established in Section \ref{proofs} using Large's localization theorem. Lastly, in Section \ref{examples} some examples are given.

\subsection*{Acknowledgments}

We thank Kristen Hendricks for many helpful conversations, and in particular for communicating to us the proof of Lemma \ref{lem:equivariant_cancellation}. We also thank Keegan Boyle for comments on a draft of this work, and Jay Patwardhan for help with the code.

\section{Background}
\label{background}
\subsection{Freely 2-periodic knots and projective knots}\label{2periodic}
Let $P$ be a freely 2-periodic knot. A freely 2-periodic knot can be written as the closure of the braid $T\circ \textrm{flip}(T)$ where $T$ is a braid word and $\textrm{flip}(T)$ is the result of rotating $T$ 180 degrees about its middle horizontal axis \cite[~Section 6.2]{manolescu2024rasmusseninvariantlinksmathbbrp3}. Thus, $P$ admits a diagram as in Figure \ref{fig:10157} where the freely 2-periodic symmetry is seen as a 180 degree rotation about an axis transverse to the plane of the diagram followed by a 180 degree rotation about an axis in the plane of the diagram. We will call such a diagram a \emph{free diagram} for $P$. The quotient knot $\overline{P}\subset \RP$ admits a diagram in $\mathbb{RP}^2$ formed by identifying antipodes of the boundary of a disk containing a diagram for $T$. 

\subsection{Knot Floer homology}\label{link}
 We assume that the reader is familiar with knot Floer homology as in \cite{OSknots,RasmussenThesis}. We briefly recall this theory here; mainly to establish notational conventions, but also to recount how the Alexander gradings work for rationally nullhomologous knots as in \cite{OSrational}. 
 
 Given a knot $K$ in a closed oriented three-manifold $Y$, a $2n$-pointed \emph{Heegaard diagram} for $K$ is a quintuple \[\mathcal{H}=(\Sigma_g, \bm{\alpha}=\{\alpha_1,\hdots,\alpha_{g+n-1}\},\bm{\beta}=\{\beta_1,\hdots,\beta_{g+n-1}\},\bm{w}=\{w_1,\hdots,w_n\},\bm{z}=\{z_1,\hdots,z_n\})\]
such that 
\begin{enumerate}
    \item each of $\bm{\alpha}$ and $\bm{\beta}$ spans a $g$ dimensional subspace of $H_1(\Sigma)$,
    \item  $\alpha_i\cap \alpha_j=\beta_i\cap \beta_j=\emptyset$ for $i\neq j$,
    \item $\alpha_i\pitchfork\beta_j$ for all $i,j$,
    \item each component of $\Sigma-\cup \alpha_i$ and $\Sigma-\cup \beta_i$ contains exactly one $\bm{w}$ point and one $\bm{z}$ point,
    \item and the union of arcs $\xi_i$ connecting each $w_i$ to a $z_j$ on $\Sigma_g-\cup_{i=1}^{g+n-1}\alpha_i$ that are then slightly pushed into the handlebody specified by $\bm{\alpha}$ and arcs $\zeta_i$ connecting each $z_j$ to a $w_i$ on $\Sigma_g-\cup_{i=1}^{g+n-1}\beta_i$ that are then slightly pushed into the handlebody specified by $\bm{\beta}$ yields the knot $K$: that is, $\bigcup_i \xi_i\cup \zeta_i=K$.
\end{enumerate}

The connected components of $\Sigma_g-\cup \alpha_i-\cup \beta_i$ are called \emph{elementary regions}. A linear combination $D$ of elementary regions is called a \emph{periodic domain} if $\partial D$ can be expressed as a linear combination of the $\alpha$ and $\beta$ curves. The Heegaard diagram $\mathcal{H}$ is \emph{weakly admissible} if every periodic domain has both positive and negative local multiplicities. From now on we assume that $\mathcal{H}$ is weakly admissible, as this is a necessary assumption to compute knot Floer homology from $\mathcal{H}.$ We define the knot Floer chain complex by
\begin{equation}\label{linkfloereq}
\widetilde{CFK}(\mathcal{H}):=CF(\talpha,\tbeta)=\bigoplus_{x\in \talpha\cap \tbeta}\F[x]
\end{equation}
where $\talpha:=\alpha_1\times\hdots\times\alpha_{g+n-1}$ and $\tbeta:=\beta_1\times\hdots\times \beta_{g+n-1}$ are Lagrangians in the manifold $\Sym^{g+n-1}(\Sigma_g\backslash (\bm{w}\cup\bm{z}))$ with respect to a suitable symplectic form as in \cite{perutz}. The differential $\partial$ on $\widetilde{CFK}(\mathcal{H})$ is then the Lagrangian intersection Floer differential which counts pseudo-holomorphic disks in $\Sym^{g+n-1}(\Sigma_g\backslash (\bm{w}\cup\bm{z}))$ with boundary on $\talpha$ and $\tbeta$. The knot Floer homology of $\mathcal{H}$ is the homology of this chain complex: \[\widetilde{HFK}(\mathcal{H}):=H_*(\widetilde{CFK}(\mathcal{H})).\] 
The version of knot Floer homology introduced above depends on the number of basepoints; if we let $V:=\F_{(0,0)}\oplus \F_{(-1,-1)}$ be the two dimensional vector space with homogeneous basis elements in bigradings $(0,0)$ and $(-1,-1)$, $\mathcal{H}'$ be a $4$-pointed Heegaard diagram for $K$, and $\mathcal{H}$ be a $2$-pointed Heegaard diagram for $K$, then 
\begin{equation}\label{HFKdependencyonbp}
    \widetilde{HFK}(\mathcal{H}')\cong\widetilde{HFK}(\mathcal{H})\otimes V^{}.
\end{equation}
The isomorphism class of the homology $\widetilde{HFK}(\mathcal{H})$ is independent of the choice of $2$-pointed diagram $\mathcal{H}$ for $K$, and hence is a knot invariant denoted $\HFKhat(Y,K)$.
\subsubsection{Gradings for rationally nullhomologous knots}\label{gradings}
We now elaborate on the gradings for rationally nullhomologous knots in knot Floer homology, beginning with a discussion of $\textrm{Spin}^\textrm{c}$ and relative $\textrm{Spin}^\textrm{c}$ structures. 

A $\textrm{Spin}^\textrm{c}$ structure on a closed oriented three-manifold $Y$ is a homology class of nowhere vanishing vector fields, where we say that vector fields are homologous if they are homotopic on the complement of a ball in $Y$; we denote the set of these by $\textrm{Spin}^\textrm{c}(Y)$. The set $\textrm{Spin}^\textrm{c}(Y)$ is an affine $H^2(Y;\Z)$ torsor via the following construction. Pick a metric on $Y$, and notice that any non-vanishing vector field $v$ can be homotoped to the unit vector field $\frac{v}{|v|}$. Now also fix a trivialization of $TY$; then there is a bijective correspondence between unit vector fields $u$ and maps $f_u:Y\to S^2$. From elementary obstruction theory we have that if $[v]=[v']$ then $f_{\frac{v}{|v|}}^*=f_{\frac{v'}{|v'|}}^*$. If $\mu$ is a generator for $H^2(S^2;\Z)$, we define \[[v]-[w]\coloneq f_{\frac{v}{|v|}}^*(\mu)-f_{\frac{w}{|w|}}^*(\mu)\in H^2(Y;\Z).\] This definition does not depend on the choice of trivialization of $TY$. The fact that $\textrm{Spin}^\textrm{c}(Y)$ is an affine $H^2(Y;\Z)$ torsor explains why there is only one $\textrm{Spin}^\textrm{c}(Y)$ structure on $S^3$, and why there are two on $\RP$ as mentioned in the introduction.

Let $\nu(K)$ denote a tubular neighborhood of $K\subset Y$, and let $\nu^\circ(K)$ denote its interior. A relative $\textrm{Spin}^\textrm{c}$ structure on $Y\backslash \nu^\circ(K)$,  is a homology class of nowhere vanishing vector fields which restrict to the canonical isotopy class of translation invariant nowhere vanishing vector fields on the toroidal boundary $\partial \nu(K)$; we denote the set of these by $\underline{\textrm{Spin}^\textrm{c}(Y,K)}$. Any generator $x\in \talpha\cap \tbeta$ is assigned a relative $\textrm{Spin}^\textrm{c}$ structure $\mathfrak{s}_{w,z}(x)$ \cite[~Section 2.3]{OSknots}. The set $\underline{\textrm{Spin}^c(Y,K)}$ is an affine torsor for $H^2(Y,K;\Z)$ in the same way that $\textrm{Spin}^\textrm{c}(Y)$ is an affine torsor for $H^2(Y;\Z)$. If a relative $\textrm{Spin}^\textrm{c}$ structure $\mathfrak{s}$ is represented by a vector field $v$, then we define the first Chern class of $\mathfrak{s}$ by $c_1(\mathfrak{s})=[v]-[-v]\in H^2(Y,K;\Z)$. Let us note the formula \begin{equation}\label{chernadd}c_1(\mathfrak{s}+\beta)=c_1(\mathfrak{s})+2\beta\end{equation} for any $\mathfrak{s}\in \underline{\textrm{Spin}^c(Y,K)}$ and $\beta\in H^2(Y,K;\Z)$. There is also a filling map $G_{Y,K}:\underline{\textrm{Spin}^\textrm{c}(Y,K)}\to \textrm{Spin}^\textrm{c}(Y)$ which has the property that for each $a\in H^2(Y,K;\Z)$ and $\mathfrak{s}\in \underline{\textrm{Spin}^\textrm{c}(Y,K)}$, \begin{equation}\label{fillequi}G_{Y,K}(\mathfrak{s}+a)=G_{Y,K}(\mathfrak{s})+\iota^*(a)\end{equation} where $\iota^*:H^2(Y,K;\Z)\to H^2(Y;\Z)$ is induced by inclusion. We refer the reader to \cite[~Section 2.3]{OSknots} for the definition of the filling map.

 Both the Alexander grading and the $\textrm{Spin}^\textrm{c}$ grading in knot Floer homology arise from the relative $\textrm{Spin}^\textrm{c}$ grading $\mathfrak{s}_{w,z}(x)$ of a generator $x\in \talpha\cap \tbeta$. The $\textrm{Spin}^\textrm{c}$ grading of $x$ is $\mathfrak{s}_w(x)\coloneq G_{Y,K}(\mathfrak{s}_{w,z}(x)).$ Now we describe the Alexander grading. Let $\mu$ be a meridian for $K$. In \cite[~Section 4.4]{MR2546619}, a formula for the Alexander grading of $x\in \talpha\cap\tbeta$ is given by \begin{equation}\label{NiAlexander}A(x)=\frac{\langle c_1(\mathfrak{s}_{w,z}(x)),[F,\partial F]\rangle -[\mu]\cdot [\partial F]}{2[\mu]\cdot [\partial F]}\end{equation} where $F$ is any rational Seifert surface. Let $B_{Y,K}\subset H^2(Y,K;\Z)$ be those cohomology classes which are the first Chern class of a relative $\textrm{Spin}^\textrm{c}$ structure for which knot Floer homology has non-zero rank in the corresponding summand. In \cite[~Proposition 6.4]{MR3122174} it is stated that for $F$ a connected rational Seifert surface,\begin{equation}\label{minchi}\min_{\alpha\in B_{Y,K}}\langle \alpha,[F,\partial F]\rangle=\chi(F).\end{equation} 

 With these ingredients in place, we can prove that the Alexander gradings of a class $1$ knot in $\RP$ lie in $\frac{1}{2}\Z+\frac{1}{4}$.
\begin{proposition}\label{gradingRP3}
    Let $\overline{P}\subset \RP$ be a class $1$ knot. For $x$ a homogeneously graded element of $\HFKhat(\RP,\overline{P})$,  $A(x)\in\frac{1}{2}\Z+\frac{1}{4}$.
\end{proposition}
\begin{proof}
 Pick a connected rational Seifert surface $F$ with $[\mu]\cdot [\partial F]=2$. Say that $c_1(\mathfrak{s})=\alpha\in B_{\RP,\overline{P}}$ achieves the minimum in Equation \ref{minchi}. Since any relative $\textrm{Spin}^\textrm{c}$ structure can be written as $\mathfrak{s}+\beta$ for some $\beta\in H^2(\RP,\overline{P};\Z)$, by Equation \ref{chernadd} we see that the numerator of Equation \ref{NiAlexander} for $\overline{P}\subset \RP$ is \[\langle c_1(\mathfrak{s}_{w,z}(x)),[F,\partial F]\rangle-2=\langle c_1(\mathfrak{s})+2\beta,[F,\partial F]\rangle-2=\chi(F)+2\langle \beta,[F,\partial F]\rangle -2.\]
This is an odd number since $\chi(F)=1-2g(F)$ because $F$ is a connected oriented surface with one boundary component.
\end{proof}
We also have the following lemma.
\begin{lemma}\label{gradingdiffs}
  Let $\overline{P}\subset \RP$ be a class $1$ knot. For $x,y$ homogeneously graded elements of $\HFKhat(\RP,\overline{P})$ if $\mathfrak{s}_w(x)=\mathfrak{s}_w(y)$ then $A(x)-A(y)\in \Z$, and if $\mathfrak{s}_w(x)\neq\mathfrak{s}_w(y)$ then $A(x)-A(y)\in \Z+\frac{1}{2}$.
\end{lemma}
\begin{proof}
By Equation \ref{fillequi} along with the fact that $\iota^*\colon H^2(\RP,\overline{P};\Z)\cong \Z\to H^2(\RP;\Z)=\Z/2\Z$ is reduction mod $2$, we see that $\mathfrak{s}_{w,z}(x)=\mathfrak{s}_{w,z}(y)+2k\gamma$ for $\gamma$ a generator of  $H^2(\RP,K;\Z)$ and $k$ some integer. Using Equation \ref{chernadd} shows
\[\langle c_1(\mathfrak{s}_{w,z}(x)),[F,\partial F]\rangle =\langle c_1(\mathfrak{s}_{w,z}(y)+2k\gamma)),[F,\partial F]\rangle=\langle c_1(\mathfrak{s}_{w,z}(y)),[F,\partial F]\rangle +4k\langle \gamma,[F,\partial F]\rangle.\] Plugging this into Equation \ref{NiAlexander}, we conclude $A(x)-A(y)\in \Z$. Similarly if $x,y$ are generators such that $\mathfrak{s}_w(x)\neq \mathfrak{s}_w(y)$, we conclude that $\mathfrak{s}_{w,z}(x)=\mathfrak{s}_{w,z}(y)+(2k+1)\gamma$ and hence $A(x)-A(y)\in \Z+\frac{1}{2}$.
\end{proof}

\section{Equivariant Heegaard diagrams for freely 2-periodic knots}
\label{diagrams}
Here we construct an equivariant Heegaard diagram $\mathcal{H}$ for $P$ and a quotient diagram $\overline{\mathcal{H}}$ for $\overline{P}$. In addition, we analyze the effect of $\tau$ on the Alexander and Maslov gradings of $\HFKhat(S^3,P)\otimes V$.

\subsection{4-pointed diagrams for freely 2-periodic knots}\label{sympring}
In \cite[~Proposition 12.1]{OSsurvey} a procedure is described for constructing the so-called ``pringle chip'' Heegaard diagram associated to a knot diagram. We slightly modify this construction to obtain a 4-pointed equivariant Heegaard diagram \[\mathcal{H}=(\Sigma, \bm{\alpha},\bm{\beta},\bm{w}=\{w_1, w_2\}, \bm{z}=\{z_1,z_2\})\] for $P$. Start with a free diagram $D$ for $P$. Let $C$ be the immersed curve obtained by forgetting about crossing data in $D$. Let $\Sigma$ be the boundary of a $\tau$-equivariant tubular neighborhood about $C$. Add the connected components of $(\R^2-C)\cap \Sigma$ as $\alpha$ curves, and for each crossing in $D$ add a $\beta$ curve in accordance with Figure \ref{fig:pringle}.
\begin{figure}
    \centering
    \includegraphics[width=0.4\linewidth]{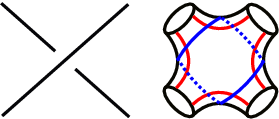}
    \caption{The placement of $\beta$ curves near crossings}
    \label{fig:pringle}
\end{figure} Add two meridians of $P$ that are interchanged by $\tau$ as $\beta$ curves, along with a pair of $w$ and $z$ basepoints on either side of each meridian that are also exchanged by $\tau$. This construction only differs from that of \cite[~Proposition 12.1]{OSsurvey} in that we include an $\alpha$ curve corresponding to the unbounded region of $\R^2-C$ in $\bm{\alpha}$, and that there are two basepointed meridians exchanged by $\tau$ instead of one. An example of a Heegaard diagram constructed in this manner starting from the diagram of $10_{157}$ on the left of Figure \ref{fig:10157} is displayed in Figure \ref{fig:10157heeg}.\begin{figure}[H]
\centering
  \includegraphics[width=.4\linewidth]{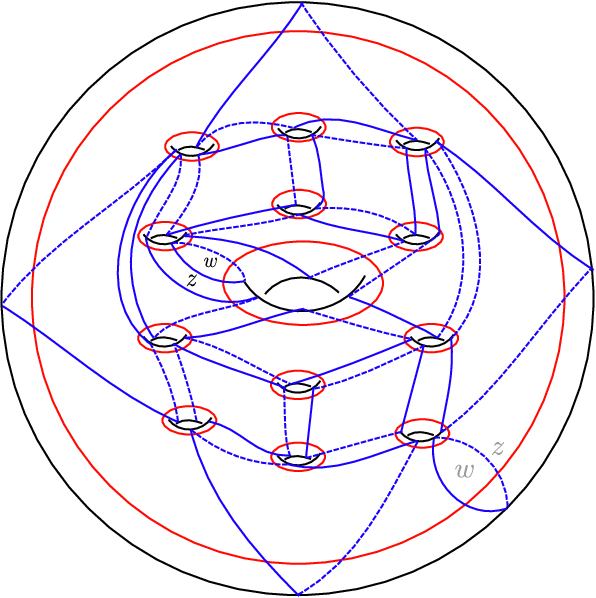}
    \caption{A 4-pointed equivariant Heegaard diagram constructed from the free diagram (Figure \ref{fig:10157}) of $10_{157}$. The $w$ and $z$ on the bottom right of the diagram are grey to indicate that they are on the bottom of the surface.}
    \label{fig:10157heeg}
\end{figure}
The quotient \[\overline{\mathcal{H}}=\mathcal{H}/\tau=(\overline{\Sigma}=\Sigma/\tau, \overline{\bm{\alpha}}=\bm{\alpha}/\tau, \overline{\bm{\beta}}=\bm{\beta}/\tau, \overline{w}=\{w_1,w_2\}/\tau,z=\{z_1,z_2\}/\tau)\]is a doubly pointed Heegaard diagram for $\overline{P}$, namely the one constructed by taking the quotient diagram in $\mathbb{RP}^2$ for $\overline{P}$ and applying the ``pringle chip'' procedure to the crossings within the disk, and identifying the boundary components of the resulting surface via a Dehn twist. See Figure \ref{fig:quotheeg} for an example of a Heegaard diagram for the quotient knot $\overline{P}$. 

\begin{figure}[H]
\centering
  \includegraphics[width=.3\linewidth]{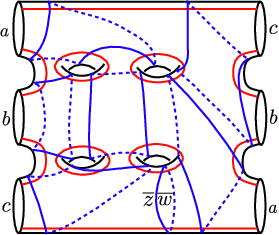}
    \caption{A two-pointed Heegaard diagram for the quotient knot $\overline{P}\subset \mathbb{RP}^3$. The boundary components with the same labels are identified via the antipodal map.}
    \label{fig:quotheeg}
\end{figure}

\subsection{The action on gradings}
The equivariant Heegaard diagram $\mathcal{H}$ of Section \ref{sympring} is invariant under $\tau\colon S^3\to S^3$, and hence there is an induced set map on the knot Floer complexes $\tau_{\#}\colon \widetilde{CFK}(\mathcal{H})\to \widetilde{CFK}(\mathcal{H})$. It should be noted that in general, the map $\tau_{\#}$ need not be a chain map. 
\begin{proposition}\label{tauongradings}
    The induced map $\tau_{\#}$ preserves Alexander and Maslov gradings.
\end{proposition}
This argument is essentially given by Levine in \cite[~Section 3.3]{MR2443111}.
\begin{proof}
  
    Pick a generator $y\in \widetilde{CFK}(\mathcal{H})$, and let $x\in \widetilde{CFK}(\mathcal{H})$ be a generator fixed by $\tau$. Letting $\pi\colon \mathcal{H}\to \overline{\mathcal{H}}$ be the restriction of the double cover $S^3\to \RP$, one can construct such a $\tau_{\#}$ equivariant generator by taking the lift $x=\pi^{-1}(\overline{x})$ of any $\overline{x}\in \widetilde{CFK}(\overline{\mathcal{H}})$. Choose a Whitney disk $\phi\in \pi_2(y,x)$. The composition $\tau\circ \phi$ is then a Whitney disk in $\pi_2(\tau_{\#}y, x)$. The Maslov indices $\mu(\phi)$ and $\mu(\tau\circ \phi)$ clearly agree by Lipshitz's combinatorial Maslov index formula \cite[~Corollary 4.10]{Lipshitzcyl}. Since $\tau$ fixes the $\bm{w}$ and $\bm{z}$, $\sum_{z\in \bm{z}} n_z(\phi)=\sum_{z\in \bm{z}}n_z(\tau\circ\phi)$ and $\sum_{w\in \bm{w}} n_w(\phi)=\sum_{w\in \bm{w}}n_w(\tau\circ\phi)$. From the formulas \cite{OSknots} for Alexander and Maslov grading differences we see then that
    \[M(x)-M(y)=M(x)-M(\tau_{\#}y)\]
    and 
    \[A(x)-A(y)=A(x)-A(\tau_{\#}y).\]
 \end{proof}

 \begin{proposition}\label{gradingupstairsdownstairs}
     If $\overline{x} \in \widetilde{CFK}(\overline{\mathcal{H}})$ is a generator with lift $x \in \widetilde{CFK}(\widetilde{\mathcal{H}})$ and $A(\overline{x}) = k  + \frac{1}{2}$, then $A(x) = 2k + \frac{1}{2}$. 
 \end{proposition}
 \begin{proof}
     This is analogous to the discussion in Section 3.2 of \cite{HendricksDP}.  Let $[\mu_{\overline{P}}^*]$ denote the Poincar\'{e} dual of the generator in $H_1(\mathbb{RP}^3 - \overline{P}; \mathbb{Q})$ represented by an oriented meridian and $[\mu_{P}^*]$ the analogous generator upstairs. The induced map $\pi_* \colon  H_1( S^3 - P ; \mathbb{Z}) \rightarrow H_1(\RP - \overline{P}; \mathbb{Z})$ is multiplication by $2$ and therefore 
     $2c_1(\mathfrak{s}_{w,z}(x)) = \pi^*c_1(\mathfrak{s}_{w,z}(\overline{x})).$ Using the relation  \[
     c_1(\underline{\mathfrak{s}}_{w,z}(x))) + \mu_{P}^*  = 2 A(x) \cdot \mu_{P}^*,
     \] the result follows: if, say, $c_1(\underline{\mathfrak{s}}_{w,z}(\overline{x})) = e[\mu_{\overline{P}}^*]$, then $c_1(\underline{\mathfrak{s}}_{w,z}(x)) = 2e[\mu_P^*]$ so
     \begin{align*}
         2e[\mu_P^*] &= (2A(x) - 1) [\mu_P^*] \textrm{ and }\\
         e[\mu_{\overline{P}}^*] &= (2A(\overline{x})-1)[\mu_{\overline{P}}^*]
     \end{align*} from which we can conclude that the lift of a generator in grading $k+1/2$ has Alexander grading $2k+1/2$.
 \end{proof}

\section{A spectral sequence for freely 2-periodic knots}\label{proofs}

Section \ref{localization} lays out hypotheses which, when met, allow for the application of Large's generalization of the Seidel--Smith localization theorem for Lagrangian Floer homology, recapped in Theorem \ref{Largelocal}. Section \ref{tniso} verifies one of these hypotheses in our setup. Section \ref{proofsection} is the proof of Theorem \ref{mainthm}.
\subsection{Polarization data and Large's localization theorem}\label{localization}
The following definitions and theorem originate from Section 3.2 of Large's paper \cite{Large}.
\begin{definition}
Let $(M,L_0,L_1$) be a symplectic manifold equipped with two Lagrangian submanifolds.
A set of \emph{polarization data for} $(M, L_0, L_1)$ is a triple $\mathfrak{p} = (E, F_0, F_1)$ where $E\twoheadrightarrow M$ is a symplectic vector bundle and $F_i$ is a Lagrangian subbundle of $E|_{L_i}$ for $i = 0, 1.$
\end{definition}
Letting $\underline{\C}$ and $\underline{\R}$ denote the trivial one-dimensional complex and real vector bundles over $M$, we may \emph{stabilize} polarization data $\mathfrak{p}=(E,F_0,F_1)$ by direct summing with $(\underline{\C},\underline{\R},i\underline{\R})$ to obtain \[\mathfrak{p}\oplus (\underline{\C},\underline{\R},i\underline{\R}):=(E\oplus \underline{\C},F_0\oplus \underline{\R},F_1\oplus i\underline{\R}).\]
\begin{definition}\label{isopol}
    Let $\mathfrak{p}=(E,F_0,F_1)$ and $\mathfrak{p}'=(E',F_0',F_1')$ be two sets of polarization data for $(M,L_0,L_1)$. An \emph{isomorphism of polarization data} is a symplectic vector bundle isomorphism $\alpha\colon E\to E'$ such that $\alpha(F_i)$ is homotopic through Lagrangian subbundles to $F_i'$ for $i=0,1.$ We say that polarization data $\mathfrak{p}$ and $\mathfrak{p}'$ are \emph{stably isomorphic} if $\mathfrak{p}\oplus (\underline{\C},\underline{\R},i\underline{\R})^{\oplus k}$ and $\mathfrak{p}'\oplus (\underline{\C},\underline{\R},i\underline{\R})^{\oplus k'}$ are isomorphic polarization data for some $k,k'\in \N.$
\end{definition}

\begin{definition}\label{tangentnormalisodef}
   Let $M$ be a symplectic manifold equipped with two Lagrangians $L_0$ and $L_1$. Suppose $\tau\colon (M,L_0,L_1)\to(M,L_0,L_1)$ is a symplectic involution with $\tau-$invariant sets \[(\M,\Lo,\Li)\subset(M,L_0,L_1).\] Define the \emph{tangent polarization} \begin{equation}
    \mathfrak{p}_T:=(T\M,T\Lo,T\Li)
\end{equation} and the \emph{normal polarization} \begin{equation}\mathfrak{p}_N:=(N\M,N\Lo,N\Li).\end{equation} Then a \emph{stable tangent normal isomorphism} is a stable isomorphism of the polarization data $\mathfrak{p}_T$ and $\mathfrak{p}_N$.
\end{definition}
The following theorem is mainly due to Large \cite{Large}. The fact that the spectral sequence splits along components of the path space $P(L_0,L_1)$ was first observed in \cite{hendricks2022rank}.
\begin{theorem}\label{Largelocal}
    Suppose that 
    \begin{enumerate}
        \item $M$ is an exact symplectic manifold and convex at infinity, and $L_0$, $L_1$ are compact exact Lagrangians.
        \item There is a symplectic involution $\tau\colon (M,L_0,L_1)\to(M,L_0,L_1)$ and an associated stable tangent normal isomorphism from $\mathfrak{p}_T$ to $\mathfrak{p}_N.$

    \end{enumerate}
Then there is a spectral sequence with $E^1$ page equal to $HF(L_0,L_1)\otimes\F[\theta,\theta^{-1}]$ and $E^\infty$ page isomorphic to $HF(\Lo, \Li)\otimes\F[\theta,\theta^{-1}]$.
\end{theorem}

\subsection{Tangent normal isomorphism}\label{tniso}
In this subsection we show that assumption $(2)$ of Theorem \ref{Largelocal} is satisfied in our setup.
\begin{lemma}\label{zfoldcover}
    The double cover $S^3 - P \rightarrow \mathbb{RP}^3 - \overline{P}$ lifts to a $\mathbb{Z}$-fold cover. 
\end{lemma}
\begin{proof}
Since $[\RP-\overline{P},S^1]\cong H^1(\mathbb{RP}^3 - \overline{P} ; \mathbb{Z}) \cong \mathbb{Z}$, there is a map $\overline{f}:\RP- \overline{P}\to S^1$ such that $\overline{f}^*:H^1(S^1;\Z)\to H^1(\RP-\overline{P};\Z)$ is an isomorphism. Applying the universal coefficient theorem, we obtain the following commutative diagram 
\begin{center}
\begin{tikzcd}
    0 \arrow{r} & H^1(\mathbb{RP}^3 - \overline{P} ; \mathbb{Z}) \arrow{r}{\cong} & \Hom( H_1(\mathbb{RP}^3 - \overline{P} , \mathbb{Z}) \arrow{r} & 0 \\
    0 \arrow{r} & \arrow{u}{\overline{f}^*} H^1(S^1, \mathbb{Z}) \arrow{r}{\cong} & \Hom(H_1(S^1), \mathbb{Z}) \arrow{r}\arrow{u}{(\overline{f}_*)^*} & 0
\end{tikzcd}
\end{center}
and therefore conclude that $\overline{f}_*:H_1(\RP-\overline P;\Z)\to H_1(S^1;\Z)$ is an isomorphism. Consider the following pullback diagram:
\begin{center}
\begin{tikzcd}
    \overline{f}^* S^1 \arrow{d}{\pi} \arrow{r}{f} & S^1 \arrow{d}{z^2} \\
    \mathbb{RP}^3 - \overline{P} \arrow{r}{\overline{f}} & S^1
\end{tikzcd}.
\end{center}
It is clear that $\overline{f}^*S^1$ lifts to a $\Z$-fold cover of $\RP-\overline{P}$, namely the pullback under $\overline{f}$ of the universal cover $\R\to S^1$. We now identify the pullback $\overline{f}^*S^1$ with $S^3 - P$. Indeed  $H^1(\mathbb{RP}^3 - \overline{P}; \mathbb{Z}_2) \cong \mathbb{Z}_2$ and therefore there is only one connected double cover of $\mathbb{RP}^3 - \overline{P}$, namely $S^3 - P$. It remains to show that $\overline{f}^*S^1$ is connected. For $z \in \mathbb{RP}^3 - \overline{P}$ let $\pi^{-1}(z) = z_1 \cup z_2$ be the fiber above $z$. Let $\gamma \subset \mathbb{RP}^3 - \overline{P}$ be a loop representing a generator of $H_1(\mathbb{RP}^3 - \overline{P})$ based at $z$. The restriction $\overline{f}|_{\gamma}$ is a degree $1$ map between circles, and hence $\overline{f}|_\gamma^*S^1\subset \overline{f}^*S^1$ contains a path connecting $z_1$ and $z_2$.

\end{proof}

Supplied with $\overline{f}$ and $f$, the proof of \cite[~Proposition 10.1]{Large}, which explicitly constructs a tangent normal isomorphism from $\overline{f}$ and $f$ in the presence of a $\Z$, cover applies verbatim to give us the following proposition.

\begin{proposition}\label{tangentnormaliso}
    The normal polarization $\mathfrak{p}$ of $Sym^g(\overline{\Sigma}-\{w,z\}) \subset Sym^{2g}(\Sigma - \{w_1, w_2, z_1, z_2 \} )$ with Lagrangians $\mathbb{T}_{\overline{\alpha}} \subset \mathbb{T}_{\alpha}$, $\mathbb{T}_{\overline{\beta}} \subset \mathbb{T}_{\beta}$ is isomorphic to the tangent polarization $(TSym^g(\Sigma - \{ w, z \}), T\mathbb{T}_{\alpha}, T\mathbb{T}_{\beta})$ of $Sym^g(\Sigma, w , z)$. 
\end{proposition}

\subsection{Proof of Theorem \ref{mainthm}}\label{proofsection}
\begin{proof}

Fix a free diagram $D$ for $P$, and let \[\mathcal{H}=(\Sigma, \bm{\alpha},\bm{\beta}, \bm{w}=\{w_1, w_2\}, \bm{z}=\{z_1,z_2\})\] be a 4-pointed $\tau$-equivariant Heegaard diagram associated to $D$ as in Section \ref{sympring}. The genus of $\Sigma$ is an odd number, say $2g-1$ for some $g\ge 1.$ Then the knot Floer homology of $\mathcal{H}$ \[\widetilde{HFK}(\mathcal{H})=HF(\mathbb{T}_{\bm{\alpha}}, \mathbb{T}_{\bm{\beta}})\] is the Lagrangian Floer homology of $L_0:=\mathbb{T}_{\bm{\alpha}}$ and $L_1:=\mathbb{T}_{\bm{\beta}}$ computed in the manifold $M:=\textrm{Sym}^{2g}(\Sigma-\{w_1,w_2,z_1,z_2\})$ endowed with an appropriate symplectic form $\omega'$ as in \cite{perutz}.

The techniques of \cite[~Section 4]{Hendricks2012rank} generalize immediately to show that $\omega'$ can be modified to a $\tau$-equivariant symplectic form $\omega\in\Omega^2(M)$ with respect to which $M$ is convex at infinity and $L_0$ and $L_1$ are still Lagrangians. Proposition 4.2 \cite{hendricks2022rank} then shows that there exists a primitive $\lambda$ of $\omega$ with $\lambda|_{L_0}=df_0$ and $\lambda|_{L_1}=df_1$ for some $f_i\in C^\infty(L_i)$ so that $L_0$ and $L_1$ are exact Lagrangians. Therefore assumption $(1)$ of Theorem \ref{Largelocal} is satisfied for our $(M, L_0, L_1).$ Assumption $(2)$ is also satisfied by Proposition \ref{tangentnormaliso}. Therefore there is a spectral sequence with $E^1$ page equal to $HF(L_0,L_1)\otimes \F[\theta,\theta^{-1}]=\HFKhat(P)\otimes V\otimes \F[\theta,\theta^{-1}]$ and $E^\infty$ page equal to $HF(\Lo,\Li).$

 The fixed point set of the involution $\tau\colon  M\to M$ defined by $\tau(x_1,...,x_{2g})=(\tau(x_1)...\tau(x_{2g}))$ is $\{(x_1\tau(x_1)...x_g\tau(x_g))\}$. Recall the quotient Heegaard diagram $\overline{\mathcal{H}}=\mathcal{H}/\tau$
for the quotient knot $\overline{P}\subset \RP$. If we define $(M',L_0',L_1'):=(\Sym^g(\Sigma-\{w,z\}), \mathbb{T}_{\overline{\bm{\alpha}}},\mathbb{T}_{\overline{\bm{\beta}}}),$ there is a natural identification $(\M, \Lo, \Li)\cong (M', L_0', L_1')$. This identification is in fact a biholomorphic map inducing an isomorphism of Lagrangian Floer homologies $HF(\Lo,\Li)\cong HF(L_0',L_1')=\HFKhat(\RP, \overline{P})$ \cite[~Appendix 1]{Hendricks2012rank}. 

Now we explain why the spectral sequence splits over Alexander gradings. Let us recap how Large's Theorem \ref{Largelocal} works in our setup. We have the chain complex $CF(L_0,L_1)$ equipped with the (knot Floer) differential $\partial$, and we may also consider the map $\tau_\#\colon CF(L_0,L_1)\to CF(L_0,L_1)$ which takes $x\in L_0\cap L_1$ to $\tau(x)\in L_0\cap L_1$. The map $\tau_\#$ is not in general a chain map. Assuming that the hypotheses of Theorem \ref{Largelocal} are satisfied, then Large shows that there is an equivariant exact Lagrangian isotopy which replaces $(L_0,L_1)$ with new Lagrangians $(\widetilde{L}_0,\widetilde{L}_1)$ and fixes the invariant sets, and also a family of complex structures $\bf{J}$ so that the map $\widetilde{\tau}_\#$ on $CF(\widetilde{L}_0,\widetilde{L}_1)$ (defined by taking $x\in \widetilde{L}_0\cap\widetilde{L}_1$ to $\tau(x)\in \widetilde{L}_0\cap\widetilde{L}_1$) is a chain map. It should be noted that even in the case that $\tau_\#$ was a chain map, $\widetilde{\tau}_\#$ is not necessarily chain homotopy equivalent to $\tau_\#$. However, the manipulation of the Lagrangians described above to obtain a chain map $\widetilde{\tau}_\#$ crucially does not affect the Alexander grading. Therefore Proposition \ref{tauongradings} allows us to conclude that $\widetilde{\tau}_\#$ preserves Alexander gradings. The spectral sequence of Theorem \ref{Largelocal} is computed from 
\begin{enumerate}
    \item the vertically filtered Leray spectral sequence of the double complex \[(CF(\widetilde{L}_0,\widetilde{L}_1)\otimes \F[\theta],\widetilde{\partial},\theta(1+\widetilde{\tau}_\#))\] -- which has $E^\infty$ page equal to $CF_{\Z/2\Z}(\widetilde{L}_0,\widetilde{L}_1)\otimes \F[\theta]$ -- followed by
    \item identification of $CF_{\Z/2\Z}(\widetilde{L}_0,\widetilde{L}_1)\otimes \F[\theta,\theta^{-1}]$ with $CF(\Lo,\Li)\otimes \F[\theta,\theta^{-1}]$ via the \emph{localization isomorphism} of \cite[~Theorem 1.1]{Large}.
\end{enumerate}
From the above discussion we see that the Leray spectral sequence of the double complex from the first item above preserves Alexander gradings, since both of its differentials do. The localization isomorphism of the second step also preserves Alexander gradings as it is constructed by counting pseudoholomorphic disks of varying index and multiplication and division by powers of $\theta$ and no flowlines pass over the basepoint divisors $V_{\overline{z}}=\overline{z}\times \Sym^{g-1}(\overline{\Sigma})$ or $V_{\overline{w}}=\overline{w}\times \Sym^{g-1}(\overline{\Sigma})$. So we have that the spectral sequence splits over Alexander gradings. The argument that Alexander grading $s$ is carried to $\frac{2s+1}{4}$ is just an application of Proposition \ref{gradingupstairsdownstairs}.
\end{proof}

\begin{figure}
    \centering
    \includegraphics[width=0.15\linewidth]{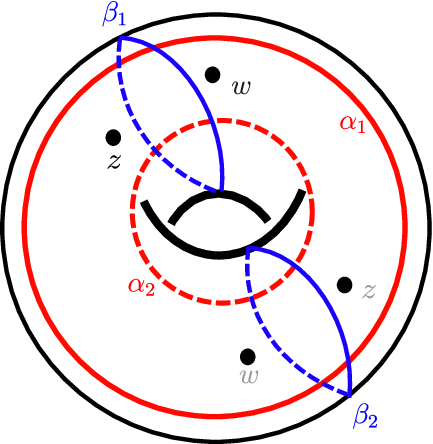}
    \caption{A four-pointed equivariant Heegaard diagram for the unknot}
    \label{fig:unknotheeg}
\end{figure}
\section{Examples}\label{examples}
\subsection{The unknot}
We provide an explicit calculation for the free 2-period on the unknot. Its quotient is the class 1 unknot in $\mathbb{RP}^3$. Using the crossingless diagram for the unknot (which is indeed a free diagram), we obtain the 4-pointed equivariant Heegaard diagram as seen in Figure \ref{fig:unknotheeg}. Since this is a \emph{nice} diagram in the sense of Sarkar-Wang \cite{SWNice}, we can compute the spectral sequence of Theorem \ref{mainthm} directly from this diagram as in \cite[~Section 5.2]{Hendricks_2016}. There are two Heegaard states which we label $a=\{\alpha_1\cap \beta_1,\alpha_2\cap\beta_2\}$ and $b=\{\alpha_1\cap \beta_2,\alpha_2\cap\beta_1\}$. As all elementary regions contain a basepoint, $E^1=\widetilde{HFK}(\mathcal{H})\otimes \F[\theta,\theta^{-1}] \cong \widetilde{CFK}(\mathcal{H})\otimes \F[\theta,\theta^{-1}]$. Now we compute $d^2$ which is the induced map $\text{Id}+ \tau_{*}$. We see that 
\[
d^2(a) = d^2(b) = 0
\]
since $\tau_*(a)=a$ and $\tau_*(b)=b$. The Alexander gradings for $a$ and $b$ are $0$ and $-1$ respectively. The higher differentials, $d^k$ for $k\ge2$, are identically 0 because of grading restrictions. The Alexander gradings $s=0$ and $s=-1$ are sent by the spectral sequence of \ref{mainthm} to Alexander gradings $\frac{1}{4}$ and $-\frac{1}{4}$, and indeed $\HFKhat(\RP, \overline{P})=\F_{(\frac{1}{4},\frac{1}{4})}\oplus \F_{(-\frac{1}{4},-\frac{1}{4})}$. 
\subsection{12n403}
Freely 2-periodic knots admit a rather simple type of equivariant Heegaard diagram: a grid diagram $\mathbb{G}$ in the sense of \cite{MOSgrids, OSSgridhomologybook} with even grid size such that the first and third quadrants are marked in the same way (we will call these A boxes), and the second and fourth quadrants are marked in the same way (B boxes). See the left hand side of Figure \ref{fig:12n403} for an example of such a grid diagram. The $\tau$ action on the grid complex is induced by swapping the A boxes with each other and the B boxes with each other. The quotient $\mathbb{G}/\tau$ of such a grid diagram is an $L(2,1)=\RP$ grid diagram for the quotient knot in $\RP$ in the sense of \cite{BGHgrid}, namely the one obtained by taking an adjacent A and B box. See the right hand side of Figure \ref{fig:12n403} for an example of such a grid diagram. \begin{figure}
    \centering

\begin{tikzpicture}[scale=0.48]

\node[gray!20, scale=8] at (2.5,2.5) {A};
\node[gray!20, scale=8] at (7.5,7.5) {A};
\node[gray!20, scale=8] at (2.5,7.5) {B};
\node[gray!20, scale=8] at (7.5,2.5) {B};

\foreach \i in {0,...,10} {
  \draw[gray!50] (\i,0) -- (\i,10);
  \draw[gray!50] (0,\i) -- (10,\i);
}

\draw[line width=1.2pt] (5,0) -- (5,10);
\draw[line width=1.2pt] (0,5) -- (10,5);

\foreach \x/\y in {
  0/0, 1/1, 2/2, 3/8, 4/9,
  5/5, 6/6, 7/7, 8/3, 9/4
} {
  \node at (\x+0.5,\y+0.5) {X};
}

\foreach \x/\y in {
  0/6, 1/7, 2/9, 3/0, 4/3,
  5/1, 6/2, 7/4, 8/5, 9/8
} {
  \node at (\x+0.5,\y+0.5) {O};
}

\end{tikzpicture}
\begin{tikzpicture}[scale=0.48]

\node[gray!20, scale=8] at (2.5,2.5) {A};
\node[gray!20, scale=8] at (7.5,2.5) {B};

\foreach \i in {0,...,10} {
  \draw[gray!50] (\i,0) -- (\i,5);
}
\foreach \j in {0,...,5} {
  \draw[gray!50] (0,\j) -- (10,\j);
}

\draw[line width=1.2pt] (5,0) -- (5,5);

\foreach \x/\y in {
  0/0, 1/1, 2/2, 8/3, 9/4
} {
  \node at (\x+0.5,\y+0.5) {X};
}

\foreach \x/\y in {
  3/0, 4/3, 5/1, 6/2, 7/4
} {
  \node at (\x+0.5,\y+0.5) {O};
}

\end{tikzpicture}
   
    \caption{Left: An equivariant grid diagram for $12n_{403}$. Right: An $\RP$ grid diagram for the quotient of $12n_{403}$.}
    \label{fig:12n403}
\end{figure}
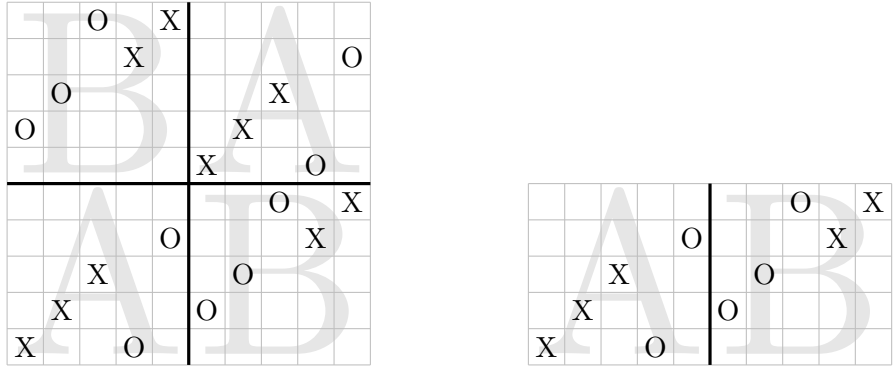
There is a spectral sequence \begin{equation}\label{gridsequence} E^1=\widetilde{GH}(\mathbb{G})\otimes\F[\theta,\theta^{-1}]\rightrightarrows E^\infty \cong\widetilde{GH}(\mathbb{G}/\tau)\otimes \F[\theta,\theta^{-1}]\end{equation} by an identical argument to the proof of Theorem \ref{mainthm}. Let $X=\F\langle xx,xy,yx,yy\rangle$ equipped with the differential $d^1(xx)=d^1(yy)=0$ and $d^1(xy)=d^1(yx)=xy+yx$, and say that $\mathbb{G}$ is a grid of size $2n$. By the methods of \cite[~Section 5.3]{Hendricks_2016}, one can further show that the spectral sequence of Equation \ref{gridsequence} is obtained from the spectral sequence of Theorem \ref{mainthm} by tensoring $E^1$ with $X^{\otimes (n-1)}$. Since grid diagrams are nice diagrams, the results of \cite[~Section 5.2]{Hendricks_2016} allow us to compute this spectral sequence algorithmically; in particular, the $d^1$ differential is given by $\theta(1+\tau_*)$ on $\widetilde{GH}(\mathbb{G})\otimes\F[\theta,\theta^{-1}]$. 

In \cite{baldwin2007computationsheegaardfloerknothomology}, Baldwin--Gillam wrote C++ code to compute grid homology. Their code builds Graph, a graph with vertices corresponding to grid states and an edge set, Edges, corresponding to the grid homology differential. Each vertex in Graph is labeled with its Maslov grading $M$; the collection of all Maslov $M$ vertices is Graph[M]. We will denote this data structure as (Graph, Edges). A method called Reduce is applied to Graph to compute $\widetilde{GH}(\mathbb{G})$; it is Gaussian elimination over $\F$. We summarize it here for the benefit of the reader. Let min and max be the minimum and maximum Maslov gradings in $\widetilde{GC}(\mathbb{G})$.
\begin{Verbatim}[frame=single, fontsize=\small]
REDUCE
input: (Graph, Edges)
for(M=max; M>min; M=M-1):
    for s -> t in Edges,
          with s in Graph[M] and t in Graph[M-1]:
        J = {j != s : j -> t in Edges}
        K = {k : s -> k in Edges}
        for j in J:
            for k in K:
                toggle j -> k in Edges
        delete all incoming and outgoing Edges to and from s,t
        delete s,t from Graph
return Graph
\end{Verbatim}
We can summarize what happens in each loop over a Maslov grading $M$ as follows. For each source-target pair $s\to t$ with $s$ in Maslov $M$, collect all of the out-vertices of $s$ into $K$ and all of the in-vertices of $t$ besides $s$ into $J$. Then make the basis changes $J\to J+s$ and $t\to \sum_{k\in K}k$ and update Edges to reflect this basis change. Delete the (now) cancelling pair $s\to t$ from Graph.

Below, we record the homological algebraic lemma which enabled us to modify Baldwin--Gillam's code to compute $\tau_*$ on $\widetilde{GH}(\mathbb{G})$.
\begin{lemma}\label{lem:equivariant_cancellation}
Suppose $(C,\partial)$ is a $\Z-$graded chain complex over $\mathbb{F}_2$ equipped with an involutive gradings preserving chain map $\tau \colon C \to C$, and suppose that, forgetting $\tau$, we may write
\[
C = D \oplus \langle s,t\rangle
\]
as chain complexes, where $\partial s=t$ and $\partial t=0$. Let $\pi \colon C \to D$ denote projection and let $i \colon D \hookrightarrow C$ denote inclusion. Define
\[
\partial'=\partial|_D, \qquad \tau'=\pi(\tau|_D).
\]
Then $(C,\partial,\tau)$ and $(D,\partial',\tau')$ are $\mathbb{Z}/2\mathbb{Z}$--equivariantly chain homotopy equivalent.
\end{lemma}

\begin{proof}
Ignoring the involutions, $(C,\partial)$ and $(D,\partial')$ are chain homotopy equivalent because the maps $\pi \colon C \to D$ and $i \colon D \to C$ are chain maps satisfying $\pi i=\mathrm{id}_D$ and $i\pi \simeq \mathrm{id}_C$. It remains to verify equivariance up to chain homotopy. Since $\tau'=\pi\tau i$, we have $\pi\tau=\tau'\pi$ on the nose. Thus it suffices to construct a map $H \colon D \to C$ satisfying $\partial H + H\partial' = \tau i + i\tau'$. Fix a graded basis for $D$, and extend it to a graded basis for $C$ by adding in $s$ and $t$. Because $\tau$ is a chain map and $\partial s=t$, the coefficient of $s$ in $\tau(x)$ agrees with the coefficient of $t$ in $\tau(\partial x)$ for every $x\in D$. Define $H \colon D \to C$ on basis elements by setting $H(x)=s$ whenever the coefficient of $t$ in $\tau(x)$ is nonzero, and $H(x)=0$ otherwise; then extend linearly. Now let $x\in D$ be a basis element. If $\tau(x)$ contains $t$, then $\partial H(x)=\partial s=t$ and $H\partial'(x)=0$ since $\tau(\partial'(x))=\partial'(\tau(x))$ doesn't contain $t$; in this case, $\tau i(x)+i\tau'(x)=\partial H(x)+H\partial'(x)=t$. If $\tau(x)$ contains $s$, then $\partial H(x)=0$ and $H\partial'(x)=s$; in this case $\tau i(x)+i\tau'(x)=\partial H(x)+H\partial'(x)=s$. Finally if $\tau(x)$ contains neither $s$ nor $t$, then $\tau i(x)+i\tau'(x)=\partial H(x)+H\partial'(x)=0$.
\end{proof}
In \cite{parikh_tau_grid_homology}, the second author modified Baldwin--Gillam's code to compute the $E^2$ page of Equation \ref{gridsequence}. A second set of edges, Edges2, corresponding to the $\tau_\#$ action is constructed. Reduce in Baldwin--Gillam's code is replaced by Reduce2 which runs Reduce while simultaneously updating Edges2 according to the basis changes $J\to J+s$ and $t\to \sum_{k\in K}k$ described in Reduce above. Using Lemma \ref{lem:equivariant_cancellation}, we notice that deleting the canceling pair $s\to t$ produces a new (Graph, Edges, Edges2) which is $\Z/2\Z$-equivariantly chain homotopy equivalent to the original. Once Reduce finishes running and Edges is empty, this means Edges2 is $\tau_*:\widetilde{GH}(\mathbb{G})\to \widetilde{GH}(\mathbb{G})$. Below we summarize this algorithm in pseudocode. It should be noted that in the actual code implementation, two graphs are constructed -- Graph and Graph2 with identical vertex sets and edge sets corresponding to the grid homology differential and $\tau_\#$ respectively -- but this is simply a cosmetic difference that made it easier to modify the existing code. The code also computes the ranks of $1+\tau_*$ homology on $\widetilde{GH}(\mathbb{G})$ in each $(M,A)$ grading, i.e., the ranks of the $E^2$ page of Equation \ref{gridsequence} using \[\dim(\ker(1+\tau_*))-\dim(\textrm{im}(1+\tau_*))=\dim(\widetilde{GH}(\mathbb{G}))-2\cdot\textrm{rank}(1+\tau_*).\] 
In \cite{parikh_tau_grid_homology} we include the full output of this program for an $8\times8$ equivariant diagram for $T(2,3)\#T(2,3)$. It displays all of the edges in Graph2 after Reduce2 has been run, listed by Maslov and Alexander grading. At the end, it contains the ranks of the $1+\tau_*$ homology of Graph2, i.e., the ranks of the $E^2$ page of Equation \ref{gridsequence}. For convenience we write below the Poincar\'{e} polynomial of the $E^2$ page according to the program output, where $q$ records the Maslov grading and $t$ records the Alexander grading: \[
q^{-7}t^{-9}
+q^{-6}t^{-8}
+4q^{-5}t^{-7}
+4q^{-4}t^{-6}
+7q^{-3}t^{-5}
+7q^{-2}t^{-4}
+7q^{-1}t^{-3}
+7t^{-2}
+4qt^{-1}
+4q^2
+q^3t
+q^4t^2 .
\]
Daniele Celoria's grid homology for lens spaces code (\cite{CeloriaGridProgram}) applied to the quotient of the grid diagram for $T(2,3)\#T(2,3)$ tells us that the Poincar\'{e} polynomial of the $E^\infty$ page of Equation \ref{gridsequence} is 

\[
\begin{aligned}
q^{-13/4}t^{-17/4}
+ q^{-11/4}t^{-15/4}
+ 4q^{-9/4}t^{-13/4}
+ 4q^{-7/4}t^{-11/4} \\
+ 7q^{-5/4}t^{-9/4}
+ 7q^{-3/4}t^{-7/4}
+ 7q^{-1/4}t^{-5/4}
+ 7q^{1/4}t^{-3/4} \\
+ 4q^{3/4}t^{-1/4}
+ 4q^{5/4}t^{1/4}
+ q^{7/4}t^{3/4}
+ q^{9/4}t^{5/4}.
\end{aligned}
\]

By inspection we see that the rank in Alexander grading $A$ of the $E^2$ page matches with the rank in Alexander grading $\frac{2A+1}{4}$ on the $E^\infty$ page, indicating that the spectral sequence of Equation \ref{gridsequence} (and hence of Theorem \ref{mainthm}) collapses on the $E^2$ page for $T(2,3)\#T(2,3)$. This behavior is expected since the spectral sequence of Theorem \ref{mainthm} will collapse on the $E^2$ page for any knot which has at most one occupied Maslov grading in each Alexander grading. An exhaustive search through $10\times10$ grid diagrams for freely 2-periodic knots yields that $12n403$ is the only freely 2-periodic knot which admits an equivariant grid diagram of size $10$ or less and has multiple occupied Maslov gradings in a single Alexander grading. The total output of our program on this grid diagram for $12n403$ along with Celoria's output on the quotient grid diagram is contained in \cite{parikh_tau_grid_homology}; the $E^2$ page and $E^\infty$ page again have the same rank. We ran this code on all $12\times 12$ and a handful of $14\times 14$ equivariant grid diagrams for possible non-$E^2$ collapse candidates. Among the examples that completed within 512 GB of memory and 72 hours on a high performance computing cluster, the spectral sequence always collapsed on the $E^2$ page. Nevertheless, we conjecture that this spectral sequence does not always collapse on $E^2$, as the analogous spectral sequence for doubly periodic knots does not always either.

\newpage

\begin{Verbatim}[frame=single, fontsize=\small]
REDUCE2
input:
    (Graph, Edges, Edges2)

for(M=max; M>min; M=M-1):
    For s -> t in Edges,
          with s in Graph[M] and t in Graph[M-1]:

        J = {j != s : j -> t in Edges}
        K = {k : s -> k in Edges}
        L = {l : s -> l in Edges2}
        N = {n : n -> t in Edges2}
        # Update Edges2 in grading M.
        # This accounts for the basis changes j -> j + s.
        
        for j in J:
            for l in L:
                toggle j -> l in Edges2

        # Update Edges2 in grading M-1.
        # This accounts for replacing t by sum_{k in K} k.
        for k in K - {t}:
            for n in N:
                toggle n -> k in Edges2
                
        # Perform Reduce
        for j in J:
            for k in K:
                toggle  j -> k in Edges
                
        delete all Edges and Edges2 arrows incident to s or t
        delete s,t from Graph
        

return the remaining part of Edges2
\end{Verbatim}

\bibliographystyle{amsalpha}
\bibliography{bib}
\end{document}